\documentclass[11pt]{article}
\usepackage[margin=1in]{geometry}
\usepackage{amssymb,amsfonts,amsmath,bbm,mathrsfs,stmaryrd}
\usepackage{xcolor}
\usepackage{url}

\usepackage{enumerate}

\usepackage[colorlinks,
             linkcolor=black!75!red,
             citecolor=blue,
             pdftitle={},
             pdfproducer={pdfLaTeX},
             pdfpagemode=None,
             bookmarksopen=true
             bookmarksnumbered=true]{hyperref}

\usepackage{tikz}
\usetikzlibrary{cd}
\usetikzlibrary{arrows,calc,decorations.pathreplacing,decorations.markings,intersections,shapes.geometric,through,fit,shapes.symbols,positioning,decorations.pathmorphing}

\usepackage{braket}
\usepackage{centernot}

\usepackage[amsmath,thmmarks,hyperref]{ntheorem}
\usepackage{cleveref}

\crefname{enumi}{}{}
\creflabelformat{enumi}{#2(#1)#3}

\crefname{section}{Section}{Sections}
\crefformat{section}{#2Section~#1#3} 
\Crefformat{section}{#2Section~#1#3} 

\crefname{subsection}{\S}{\S\S}
\crefformat{subsection}{#2\S#1#3} 
\Crefformat{subsection}{#2\S#1#3} 

\theoremstyle{plain}

\newtheorem{lemma}{Lemma}[section]
\newtheorem{proposition}[lemma]{Proposition}
\newtheorem{corollary}[lemma]{Corollary}
\newtheorem{theorem}[lemma]{Theorem}

\theoremstyle{nonumberplain}

\theoremstyle{plain}
\theorembodyfont{\upshape}
\theoremsymbol{\ensuremath{\blacklozenge}}

\newtheorem{definition}[lemma]{Definition}

\newtheorem{remark}[lemma]{Remark}

\crefname{definition}{definition}{definitions}
\crefformat{definition}{#2definition~#1#3} 
\Crefformat{definition}{#2Definition~#1#3} 

\crefname{ex}{example}{examples}
\crefformat{example}{#2example~#1#3} 
\Crefformat{example}{#2Example~#1#3} 

\crefname{remark}{remark}{remarks}
\crefformat{remark}{#2remark~#1#3} 
\Crefformat{remark}{#2Remark~#1#3} 

\crefname{convention}{convention}{conventions}
\crefformat{convention}{#2convention~#1#3} 
\Crefformat{convention}{#2Convention~#1#3}

\crefname{lemma}{lemma}{lemmas}
\crefformat{lemma}{#2lemma~#1#3} 
\Crefformat{lemma}{#2Lemma~#1#3} 

\crefname{proposition}{proposition}{propositions}
\crefformat{proposition}{#2proposition~#1#3} 
\Crefformat{proposition}{#2Proposition~#1#3} 

\crefname{corollary}{corollary}{corollaries}
\crefformat{corollary}{#2corollary~#1#3} 
\Crefformat{corollary}{#2Corollary~#1#3} 

\crefname{theorem}{theorem}{theorems}
\crefformat{theorem}{#2theorem~#1#3} 
\Crefformat{theorem}{#2Theorem~#1#3} 

\crefname{assumption}{assumption}{Assumptions}
\crefformat{assumption}{#2assumption~#1#3} 
\Crefformat{assumption}{#2Assumption~#1#3} 

\crefname{equation}{}{}
\crefformat{equation}{(#2#1#3)} 
\Crefformat{equation}{(#2#1#3)}

\theoremstyle{nonumberplain}
\theoremsymbol{\ensuremath{\blacksquare}}

\newtheorem{proof}{Proof}
\newcommand\pf[1]{\newtheorem{#1}{Proof of \Cref{#1}}}

\newcommand\bF{{\mathbb F}}

\newcommand\bN{{\mathbb N}}

\newcommand\bZ{{\mathbb Z}}

\newcommand\cI{{\mathcal I}}

\newcommand\cM{{\mathcal M}}
\newcommand\cN{{\mathcal N}}

\DeclareMathOperator{\id}{id}

\title{Projective discrete modules over profinite groups}
\author{Alexandru Chirvasitu and Ryo Kanda}

\begin{document}

\date{}

\newcommand{\Addresses}{{
  \bigskip
  \footnotesize

  \textsc{Department of Mathematics, University at Buffalo, Buffalo,
    NY 14260-2900, USA}\par\nopagebreak \textit{E-mail address}:
  \texttt{achirvas@buffalo.edu}

  \medskip

  \textsc{Department of Mathematics, Graduate School of
  Science, Osaka University, Toyonaka, Osaka, 560-0043, Japan}\par\nopagebreak \textit{E-mail address}:
  \texttt{ryo.kanda.math@gmail.com}

}}

\maketitle

\begin{abstract}
  We show that the category of discrete modules over an infinite profinite group has no non-zero projective objects and does not satisfy Ab4*. We also prove the same types of results in a generalized setting using a ring with linear topology.
\end{abstract}


\vspace{.5cm}


\tableofcontents

\section*{Introduction}

The starting point for the present note is the pervasive phenomenon whereby abelian categories encountered ``in nature'', housing various (co)homology theories, tend to have enough injectives but much more rarely have enough projectives.

This is familiar, for instance, for various flavors of sheaves:
\begin{itemize}
\item According to \cite[Exercise~III.6.2]{hrt} the categories of modules, coherent modules and quasi-coherent modules on projective lines over infinite fields (with the Zariski topology) do not have enough projectives.
\item For a more general class of schemes, it is shown in \cite[Theorem~1.1]{Kanda} that the category of quasi-coherent sheaves on a non-affine divisorial noetherian scheme (e.g.\ a non-affine quasi-projective schemes over a commutative noetherian ring) does not have enough projectives.
\item On a slightly different note, locally connected Hausdorff spaces with no isolated points admit no non-zero projective sheaves of abelian groups \cite[p.~30, Exercise~4]{bred}.
\end{itemize}

Here, we examine the category of {\it discrete modules} over a profinite group $G$, used in defining the cohomology groups $H^i(G,-)$ (see, e.g., \cite[\S 2]{ser-gal} for background and \Cref{se.prel} below for definitions).

Our main results show that if $G$ is infinite then the category of discrete modules
\begin{itemize}
\item admits no non-zero projective object (\Cref{th.iff-fin}) and
\item does not satisfy Grothendieck's \textnormal{Ab4*} condition (\Cref{pr.not-ab4}), that is, products are not exact.
\end{itemize}
Either of the two properties implies that the category does not have enough projectives.

We also consider the analogous question over fields (rather than the integers), resulting in a characterization of those profinite groups for which the category of discrete modules has enough (or equivalently, non-zero) projectives in characteristic $p$ (\Cref{th.fld}): they are exactly those whose Sylow $p$-subgroups in the sense of \cite[\S 1.4]{ser-gal} are finite. This will also provide a characterization of projectives in the said category (\Cref{cor.char-proj}).

A natural question is whether the techniques used in the proofs of the main results can be applied to more general Grothendieck categories. The discrete modules over a profinite group form a prelocalizing subcategory of the category of all modules, so one reasonable attempt is to generalize them to a prelocalizing subcategory of the category of modules over a ring. It is known that such subcategories bijectively correspond to linear topologies of the ring (see \cite[Proposition~VI.4.2]{Stenstrom}). We will express a necessary condition in terms of a linear topology, and prove a generalized results for those prelocalizing subcategories satisfying the condition (\Cref{thm.ring}).

\subsection*{Acknowledgements}

A.C. was partially supported by NSF grant DMS-1801011. R.K. was supported by JSPS KAKENHI Grant Numbers JP17K14164 and JP16H06337.

\section{Preliminaries}\label{se.prel}

Let $G$ be a profinite group. We denote by $(\cN,\le)$ the poset of open normal subgroups of $G$, ordered by inclusion.

\begin{definition}
  A {\it discrete $G$-module} is a (left) $G$-module $M$ with the property that the action $G\times M\to M$ is continuous when $M$ is equipped with the discrete topology.
	
  We write ${}_G^d\mathrm{Mod}$ for the category of discrete $G$-modules, ${}_{G}\mathrm{Mod}$ for the category of all $G$-modules (without topology). Note that ${}_{G}\mathrm{Mod}$ is canonically equivalent to the category of left modules over the group algebra $\bZ[G]$.
\end{definition}

\begin{remark}\label{rem.disc}
  Discrete modules can be characterized as those $G$-modules $M$ with the property that
  \begin{equation}\label{eq:1}
    M=\varinjlim_{H\in\cN} M^H
  \end{equation}
  where $M^{H}$ is the submodule of $M$ consisting of those elements $x$ that are fixed by every element of $H$ (see \cite[Proposition~6.1.2]{Wilson}).
  The direct limit can be replaced by a sum or a union.
  
  Since ${}_G^d\mathrm{Mod}$ is a full subcategory of the Grothendieck category ${}_{G}\mathrm{Mod}$ closed under subobjects, quotient objects, and coproducts (such a subcategory is called \emph{prelocalizing}, \emph{weakly closed}, or a \emph{hereditary pretorsion class}), it is also a Grothendieck category. The forgetful functor ${}_G^d\mathrm{Mod}\to{}_{G}\mathrm{Mod}$ has a right adjoint that sends a $G$-module $M$ to its largest discrete submodule $\varinjlim_{H} M^{H}$.
  
  The adjoint property tells us how to compute inverse limits in ${}_G^d\mathrm{Mod}$. Let $\{M_{i}\}_{i\in I}$ is an inverse system in ${}_G^d\mathrm{Mod}$ and denote by $\varprojlim_{i}M_{i}$ the inverse limit taken in ${}_G^d\mathrm{Mod}$ and by $\varprojlim^f_i M_{i}$ the one taken in ${}_G\mathrm{Mod}$ (which is simply the limit of corresponding abelian groups; the `$f$' superscript stands for ``full''). Then
  \begin{equation*}
    \varprojlim_{i}M_{i}=\varinjlim_{H\in\cN}\bigg(\varprojlim^f_i M_{i}\bigg)^{\!H},
  \end{equation*}
  which is the largest discrete submodule of $\varprojlim^f_i M_{i}$.
\end{remark}

We will refer briefly to coalgebras and comodules over fields, for which our background reference is \cite[Chapters 1 and~2]{dnr}.

Write $\cM^C$ for the category of right $C$-comodules. According to \cite[Corollary~2.4.21]{dnr}, $\cM^C$ has enough projectives if and only if every finite-dimensional $C$-comodule has a (finite-dimensional again) projective cover. Furthermore, according to the proof of \cite[Corollary~2.4.22]{dnr} every projective object in $\cM^C$ is a direct sum of finite-dimensional projective objects. These observations will be of use later.

\section{Main results}\label{se.main}

\begin{theorem}\label{th.iff-fin}
Let $G$ be a profinite group. The category ${}_G^d \mathrm{Mod}$ has non-zero projective objects if and only if $G$ is finite.
\end{theorem}


\begin{proof}
  One implication is obvious, so we assume that $G$ is an infinite profinite group and that $P$ is a nonzero projective object in ${}_G^d \mathrm{Mod}$. $P$ can be expressed as a union of its submodules $P^H$ as in \Cref{eq:1}, and hence we have a surjection
  \begin{equation*}
    \bigoplus_{H\in \cN}P^H\to P. 
  \end{equation*}
  In turn, every $P^H$ is surjected upon by some free $G/H$-module $F_H$, which is a direct sum of copies of $\bZ[G/H]$, and we have an epimorphism
  \begin{equation}\label{eq:2}
    F:=\bigoplus_{H\in\cN} F_H\to P. 
  \end{equation}

  Now let $H_0>H_1>\cdots$ be a strictly descending sequence of groups in $\cN$ (one exists, since $G$ is assumed infinite). For each non-negative integer $i$ we construct a surjection $E_i\to F$ defined by substituting $\bZ[G/(H\cap H_i)]$ for every $\bZ[G/H]$ summand of $F$ and surjecting
  \begin{equation*}
    \bZ[G/(H\cap H_i)]\to \bZ[G/H]
  \end{equation*}
  naturally.

  We now claim that the projectivity of $P$ entails a factorization 
  \begin{equation}\label{eq:4}
    \begin{tikzpicture}[auto,baseline=(current  bounding  box.center)]
      \path[anchor=base] (0,0) node (pl) {$P$} +(2,.5) node (lim) {$\varprojlim_i E_i$} +(4,.5) node (f) {$F$} +(6,0) node (pr) {$P$};
      \draw[->] (pl) to[bend right=6] node[pos=.5,auto,swap] {$\scriptstyle \id$} (pr);
      \draw[->] (pl) to[bend left=6] node[pos=.5,auto] {$\scriptstyle$} (lim);
      \draw[->] (lim) to[bend left=6] node[pos=.5,auto] {$\scriptstyle$} (f);
      \draw[->] (f) to[bend left=6] node[pos=.5,auto] {$\scriptstyle$} (pr);
    \end{tikzpicture}
  \end{equation}
  where the limit is taken in the category ${}^d_G\mathrm{Mod}$.

  To see this, note first that $P$ splits off as a summand of $F$. Since $E_0$ surjects onto the latter, we further obtain an direct summand embedding $P\to E_0$; now repeat the procedure to lift this to a map $P\to E_1$ fitting into a triangle
  \begin{equation*}
    \begin{tikzpicture}[auto,baseline=(current  bounding  box.center)]
      \path[anchor=base] (0,0) node (pl) {$P$} +(2,.5) node (lim) {$E_1$} +(4,0) node (f) {$E_0$};
      \draw[->] (pl) to[bend right=6] node[pos=.5,auto,swap] {$\scriptstyle $} (f);
      \draw[->] (pl) to[bend left=6] node[pos=.5,auto] {$\scriptstyle$} (lim);
      \draw[->] (lim) to[bend left=6] node[pos=.5,auto] {$\scriptstyle$} (f);
    \end{tikzpicture}
  \end{equation*}
Continuing this recursively will produce \Cref{eq:4}. 

The contradiction will follow if we show that the limit $\varprojlim_i E_i$ in \Cref{eq:4} vanishes. As explained in \Cref{rem.disc}, the limit in ${}_G^d\mathrm{Mod}$ is obtained as the largest discrete submodule of the limit $\varprojlim^{f}_i E_i$ in ${}_G\mathrm{Mod}$.
We thus have to argue that $\varprojlim_i^f E_i$ contains no non-zero elements fixed by every open normal subgroup $H\trianglelefteq G$. To see this, recall that the connecting morphisms
\begin{equation*}
  E_{i+1}\to E_i,\quad i\in \bN
\end{equation*}
whose filtered limit we are taking are coproducts of copies of the standard epimorphisms
\begin{equation*}
  \bZ[G/(H\cap H_{i+1})]\to   \bZ[G/(H\cap H_{i})]
\end{equation*}
for the summands $\bZ[G/H]$ of $F$. To illustrate the claimed vanishing of the limit without irrelevant notational overhead we will consider the simpler limit
\begin{equation*}
  \varprojlim_i^f \bZ[G/(H\cap H_i)]
\end{equation*}
along the canonical surjections.

For every $K\in \cN$ the image of the $K$-invariants through
\begin{equation*}
  \bZ[G/(H\cap H_{j})]\to   \bZ[G/(H\cap H_{i})],\quad j>i
\end{equation*}
consists of $[K\cap H\cap H_{i}:K\cap H\cap H_{j}]$-multiples in the latter free abelian group. Indeed, if we write $N_{i}=H\cap H_{i}$ and $N_{j}=H\cap H_{j}$, then an element of $\bZ[G/N_{j}]^{K}$ is of the form
\begin{equation*}
	\sum_{gN_{j}\in G/N_{j}}n_{gN_{j}}\cdot gN_{j}
\end{equation*}
such that $n_{gN_{j}}=n_{g'N_{j}}$ whenever $gKN_{j}=g'KN_{j}$. In particular, the elements of $G/N_{j}$ in the same coset of $(K\cap N_{i})N_{j}$ have the same coefficient, and they are sent to a single element of $G/N_{i}$ whose coefficient is a multiple of
\begin{equation*}
	[(K\cap N_{i})N_{j}:N_{j}]=[K\cap N_{i}:K\cap N_{j}]=[K\cap H\cap H_{i}:K\cap H\cap H_{j}].
\end{equation*}

Since $[G:K\cap H]<\infty$, the sequence $\{K\cap H\cap H_{j}\}_{j=0}^{\infty}$ is strictly descending. Thus the index $[K\cap H\cap H_{i}:K\cap H\cap H_{j}]$ grows indefinitely with $j$ for fixed $i$. It follows that the image of
\begin{equation*}
  \left(\varprojlim_i^f \bZ[G/(H\cap H_i)]\right)^K
\end{equation*}
in every $\bZ[G/(H\cap H_i)]$ vanishes. In conclusion, as claimed, the maximal discrete submodule
\begin{equation*}
  {}_G^d\mathrm{Mod}\ \ni\ \varprojlim_i \bZ[G/(H\cap H_i)]\ \subseteq\ \varprojlim_K^f \bZ[G/(H\cap H_i)]\ \in\ {}_G\mathrm{Mod}
\end{equation*}
is trivial.
\end{proof}

\Cref{th.iff-fin} shows in particular that for infinite $G$ the category of discrete $G$-modules fails to have enough projective modules. That failure is in fact even stronger: recall that for a Grothendieck category, having enough projectives implies
\begin{itemize}
\item having non-zero projectives and  
\item being \textnormal{Ab4*}, i.e. having exact products.  
\end{itemize}

The next result proves that the second condition is also violated in the present context:

\begin{proposition}\label{pr.not-ab4}
Let $G$ be a profinite group. The category ${}_G^d \mathrm{Mod}$ satisfies \textnormal{Ab4*} if and only if $G$ is finite.
\end{proposition}
\begin{proof}
Again, one implication is obvious, so we assume that $G$ is infinite. We will show that the product of the augmentation morphisms 
  \begin{equation}\label{eq:6}
    \bZ[G/H]\to \bZ[G/G]=\bZ
  \end{equation}
  for $H\in \cN$ is not epic. Indeed, for fixed $K\in \cN$ the image of $\bZ[G/H]^K$ through \Cref{eq:6} is contained in the ideal generated by $[K:H]$ whenever $H\le K$. It follows that the all-$1$ element of the product $\prod_{H\in\cN} \bZ$ is not contained in the image of $\prod_{H\in \cN}\bZ[G/H]$, proving the claim.
\end{proof}

\section{Ground fields}\label{se.fld}

The situation is rather different when working over a field $k$ in place of $\bZ$. First, recall the notion of {\it supernatural number} from \cite[\S 1.3]{ser-gal}: simply a formal product of the form
\begin{equation*}
  \prod_{\text{primes } p}p^{n_p},\quad n_p\in \bZ_{\ge 0}\cup\{\infty\}. 
\end{equation*}
A profinite group $G$ has an order $|G|$, well-defined as a supernatural number as the least common multiple of all orders $G/H$ for $H\in \cN$. Similarly, we can define the index $[G:H]$ as a supernatural number for every closed subgroup $H\le G$. There is also a concept of {\it Sylow $p$-subgroup} of $G$, i.e. a closed subgroup whose supernatural order is of the form $p^{n_p}$ (meaning it is {\it pro-$p$}, i.e. a filtered limit of finite $p$-groups) and whose index in $G$ does not have $p$ as a factor (or is coprime to $p$, in short). We refer to \cite[Chapter~1]{ser-gal} for details.

Now let $k$ be a ground field of characteristic $p$ (a prime or zero) and write ${}_G^d\mathrm{Vect}$ for the category of discrete $G$-modules over $k$ (note that we are suppressing $k$ from the notation, for brevity). The main result of the present section is a characterization of those $G$ for which this category admits non-zero projectives.

\begin{theorem}\label{th.fld}
  For a profinite group $G$ and a field $k$ of characteristic $p$ the following conditions are equivalent:
  \begin{enumerate}[(a)]
  \item\label{item:1} The category ${}_G^d\mathrm{Vect}$ has a non-zero projective object.
  \item\label{item:2} The category ${}_G^d\mathrm{Vect}$ has enough projective objects.
  \item\label{item:3} The characteristic $p$ of $k$ has finite exponent in the supernatural number $|G|$.    
  \end{enumerate}
\end{theorem}

\begin{remark}
  The condition in \labelcref{item:3} is by convention assumed to hold vacuously when the characteristic is zero.
\end{remark}

Before settling into the proof proper we make the preliminary observation that the category ${}^d_G\mathrm{Vect}$ is nothing but the category $\cM^C$ of comodules over the $k$-coalgebra $C=k(G)$ of continuous $k$-valued functions on $G$ (where $k$ is equipped with the discrete topology).

\pf{th.fld}
\begin{th.fld}
  That \labelcref{item:1} follows from \labelcref{item:2} is obvious, so we focus on proving \labelcref{item:3} $\Rightarrow$ \labelcref{item:2} and \labelcref{item:1} $\Rightarrow$ \labelcref{item:3}.

  \vspace{.5cm}

  {\bf \labelcref{item:1} $\Rightarrow$ \labelcref{item:3}.} As noted at the end of \Cref{se.prel}, the existence of a non-zero projective entails the existence of a finite-dimensional one ($P$, say). $P$ will then be projective over some group algebra $k[G/H]$ for $H\in \cN$, and we can furthermore assume that it is a summand of $k[G/H]$ (because it can be written as a direct sum of indecomposable projectives, which are summands of $k[G/H]$).

  Now, if \labelcref{item:3} were false then we could find a subgroup $K\in \cN$ of $H$ with $p$ dividing $[H:K]$. Now consider the projection
  \begin{equation}\label{eq:5}
    k[G/K]\to k[G/H]\to P.
  \end{equation}
  Splitting it would provide a summand $\cong P$ of $k[G/K]$ acted upon trivially by $H/K$. But then, since $p$ divides the order of this latter group, the image of this summand through \Cref{eq:5} must vanish. This gives a contradiction and proves the desired implication.
   
  \vspace{.5cm}
  
  {\bf \labelcref{item:3} $\Rightarrow$ \labelcref{item:2}.} By \cite[Corollary~2.4.21]{dnr} it suffices to show that every {\it finite-dimensional} discrete module $M$ admits a surjection by a projective. We can regard $M$ as a module over $G/H$ for some $H\in \cN$, and suppose $H$ is small enough to ensure that $|G/H|$ and $|G|$ have the same $p$-exponent.

  Naturally, $M$ is a quotient of a finite-dimensional projective $P$ over $G/H$. It remains to argue that $P$ is still projective over $G$, which will be the goal for the remainder of the proof.

  Since $P$ is $G/H$-projective, it must be projective over a Sylow $p$-subgroup $\overline{S}\le G/H$. Our choice of $H$ (such that $[G:H]$ is divisible by the same exact power of $p$ as $|G|$) means that there is a Sylow $p$-subgroup $S$ of $G$ mapping isomorphically over $\overline{S}$. We thus know that the restriction of $P$ to $S$ is projective.

  Projectivity over $G$ means showing that all higher cohomology
  \begin{equation*}
    \mathrm{Ext}^i(P,-)\cong H^i(G,-\otimes P^*),\ i\ge 1 
  \end{equation*}
  in the category ${}^d_G\mathrm{Vect}$ vanishes. We already know that it vanishes upon restricting via
  \begin{equation*}
    \mathrm{res}:H^i(G,-\otimes P^*)\to H^i(S,-\otimes P^*),
  \end{equation*}
  and the conclusion follows from the fact that this restriction morphism is one-to-one: this is more or less \cite[\S 2.4, Corollary to Proposition~9]{ser-gal}. Although the latter result refers to cohomology over $\bZ$, the techniques apply essentially verbatim over $k$.

  \vspace{.5cm}

  The above reasoning applies unequivocally in positive characteristic, but requires interpretation in characteristic zero. In that case the $p$ referred to throughout will be $0$, Sylow subgroups will be trivial, etc. The validity of the proof will not be affected if these obvious modifications are made, using the fact that in characteristic zero the higher cohomology of a profinite group is
  \begin{equation*}
    H^i(G,-) = \varinjlim_{H\in \cN}H^i(G/H,-)=\varinjlim 0 = 0,
  \end{equation*}
  i.e. the coalgebra $k(G)$ is {\it cosemisimple}.
\end{th.fld}

We note in passing that as a byproduct of the proof of \Cref{th.fld} we obtain the following characterization of projectives in ${}^d_G\mathrm{Vect}$:

\begin{corollary}\label{cor.char-proj}
  Let $G$ be a profinite group and $k$ a field of characteristic $p$. The following statements hold:
  \begin{enumerate}[(a)]
  \item\label{item:4} For every open normal subgroup $H$ with $p\centernot\mid |H|$, a $G/H$-module is projective if and only if it is projective over $G$.
  \item\label{item:5} The projective objects in ${}^d_G\mathrm{Vect}$ are direct sums of finite-dimensional, indecomposable projectives over $G/H$ for $H$ ranging over the open normal subgroups of $G$ with $p\centernot\mid |H|$.
  \end{enumerate}
\end{corollary}

\begin{remark}
  If the characteristic is zero then the class of subgroups $H$ in \Cref{cor.char-proj} is unrestricted, i.e. we range over {\it all} of $\cN$.
\end{remark}

Part \Cref{item:4} of \Cref{cor.char-proj} has the following partial converse.

\begin{lemma}\label{le.res-to-syl}
Let $G$ be a profinite group, $k$ a field of characteristic $p$ and $H\le G$ a closed normal subgroup with $p\mid |H|$. Then, non-trivial $G/H$-modules cannot be projective over $G$. 
\end{lemma}
\begin{proof}
  Let $P$ be a hypothetical non-zero $G/H$-module projective over $G$ and let $T\le H$ be a (necessarily non-trivial) Sylow $p$-subgroup.

  The restriction functor
  \begin{equation}\label{eq:9}
    \mathrm{res}:{}^d_G\mathrm{Vect}\to {}^d_T\mathrm{Vect} 
  \end{equation}
  is left adjoint to the exact functor
  \begin{equation*}
    {}^d_T\mathrm{Vect} \ni M\mapsto\mathrm{Map}_T(G,M) \in {}^d_G\mathrm{Vect},
  \end{equation*}
  where
  \begin{itemize}
  \item $\mathrm{Map}_T$ denotes continuous $T$-equivariant maps $G\to M$ with $M$ equipped with the discrete topology;
  \item $T$-equivariant means $f(tg)=tf(g)$ for all $t\in T$ and $g\in G$;
  \item the $G$-action is given by $g\triangleright f = f(\bullet\, g)$;
  \end{itemize}
  see \cite[\S 6.10]{rz}. Since it has an exact right adjoint, it follows that \Cref{eq:9} preserves projectivity. In particular, $P$ will be projective over $T$. This, however, is impossible: $T$ is a pro-$p$-group, which \Cref{cor.char-proj} \Cref{item:5} ensures cannot have non-zero discrete projectives in characteristic $p$.
\end{proof}

We now have the following alternate take on \Cref{th.iff-fin}:

\pf{th.iff-fin}
\begin{th.iff-fin}[alternative]
  We begin as before, assuming that $G$ is an infinite profinite group and $P$ a non-zero projective object in ${}^d_G\mathrm{Mod}$. We then have the epimorphism \Cref{eq:2} onto $P$ from a direct sum $F=\bigoplus_{H\in \cN}F_H$ of free $G/H$-modules.

  By projectivity the epimorphism $P$ splits, realizing $P$ as a summand of $F$. In particular, there is some finite multiset of groups $H_i\in \cN$ so that
  \begin{equation}\label{eq:8}
    \{0\}\ne \left(\bigoplus_i \bZ[G/H_i]\right)\cap P\subset F.
  \end{equation}

  Now apply the scalar-extension functor
  \begin{equation*}
    \mathrm{E}_p:=\bF_p\otimes_{\bZ}-:{}^d_G\mathrm{Mod}\to {}^d_G\mathrm{Vect}
  \end{equation*}
  for a finite field $\bF_p$ with $p$ elements, with $p$ chosen judiciously (more on this below). The functor preserves projectivity (because it is left adjoint to the exact scalar restriction functor), so $\mathrm{E}_p(P)$  is projective (and clearly non-zero, since it was obtained by scalar-extending a free abelian group).

  We now have direct sum decomposition
  \begin{equation}\label{eq:7}
    \mathrm{E}_p(P)\oplus * \cong \mathrm{E}_{p}(F)=\bigoplus_{H\in\cN} \bF_p[G/H]^{\oplus}.
  \end{equation}
  The summands on the right-most side can be further decomposed as finite direct sums of modules with local endomorphism rings (because they are modules over finite group algebras over $\bF_{p}$). 

  \Cref{eq:8} implies its counterpart over $\bF_p$:
  \begin{equation*}
    \{0\}\ne \left(\bigoplus_i \bF_p[G/H_i]\right)\cap (P/pP)\subset F/pF.
  \end{equation*}
  The proof of \cite[Theorem~26.5]{af} applied to $M=F/pF$, $K=P/pP$, and $N=\bigoplus_{i}\bF_{p}[G/H_{i}]$ shows that $P/pP$ has a non-zero summand $H$ isomorphic to a summand of $N$. Thus \cite[Corollary~26.6]{af} implies that $P/pP$ has a summand, say $S$, isomorphic to 
  an indecomposable summand of $\bF_p[G/H_i]$ for one of the finitely many $i$ in \Cref{eq:8}. Note that $S$ is projective in ${}^d_G\mathrm{Vect}$, being a summand of the projective object $P/pP$.

  Now we can specialize $p$: choose it so as to ensure that it divides the (supernatural) order of $H=\bigcap_i H_i$ (this is possible, since the latter group has finite index in the infinite profinite group $G$). The projectivity of $S$ over $G$ contradicts \Cref{le.res-to-syl}, finishing the proof.
\end{th.iff-fin}


\section{Ring-theoretic interpretation}

\Cref{th.iff-fin} and \Cref{pr.not-ab4} can be shown in a more general setting of a ring with a filter of ideals.

Let $R$ be a ring. We denote by ${}_{R}\mathrm{Mod}$ the category of (left) $R$-modules.

Let $\cI$ be a downward filtered set of (two-sided) ideals of $R$, that is, $\cI$ is a non-empty set of ideals and for any $I_{1},I_{2}\in\cI$, there exists $J\in\cI$ such that $J\subseteq I_{1}$ and $J\subseteq I_{2}$. For ease of comparison with the previous section, we use the notation
\begin{equation*}
	{}_{R}^{g}\mathrm{Mod}:=\{M\in{}_{R}\mathrm{Mod}\mid M=\varinjlim_{I\in\cI}M_{I}\}
\end{equation*}
where
\begin{equation*}
	M_{I}:=\{x\in M\mid Ix=0\}.
\end{equation*}
The forgetful functor ${}_{R}^{g}\mathrm{Mod}\to{}_{R}\mathrm{Mod}$ admits a right adjoint given by $M\mapsto\varinjlim_{I}M_{I}$. Products in ${}_{R}^{g}\mathrm{Mod}$ are described in a similar way to \Cref{rem.disc}.

In view of \cite[Proposition~VI.4.2]{Stenstrom}, ${}_{R}^{g}\mathrm{Mod}$ is the prelocalizing subcategory of ${}_{R}\mathrm{Mod}$ corresponding to a left linear topology of $R$ that admits a fundamental system of neighborhoods of $0\in R$ consisting of two-sided ideals.

\begin{remark}\label{rem.ex.profin}
	For a profinite group $G$, we take the group algebra $R:=\bZ[G]$ without topology. The poset $\cN$ of open normal subgroups of $G$ defines the downward filtered set of ideals
	\begin{equation*}
		\cI:=\{I_{H}\mid H\in\cN\}
	\end{equation*}
	where $I_{H}$ is the kernel of the canonical surjection $\bZ[G]\to\bZ[G/H]$. For each $M\in {}_{G}\mathrm{Mod}={}_{R}\mathrm{Mod}$, the largest submodule of $M$ belonging to ${}_{G/H}\mathrm{Mod}$ is $M^{H}$ as well as $M_{I_{H}}$. Thus $M^{H}=M_{I_{H}}$. By the characterization \Cref{eq:1} in \Cref{rem.disc}, we have ${}_{G}^{g}\mathrm{Mod}={}_{R}^{g}\mathrm{Mod}$.
\end{remark}

For two ideals $I,J\subseteq R$, define the ideal
\begin{equation*}
	(J:I):=\{r\in R\mid Ir\subseteq J\}.
\end{equation*}
Note that $(R/J)_{I}=(J:I)/J$.

\begin{theorem}\label{thm.ring}
  Let $R$ be a ring and $\cI$ a downward filtered set of ideals of $R$. Suppose that for each $I\in\cI$,
  \begin{equation}\label{eq.cond.ideal}
    \bigcap_{J}((J:I)+I)=I
  \end{equation}
  where $J$ runs over all ideals in $\cI$ with $J\subseteq I$. Then the following hold:
  \begin{enumerate}[(a)]
  \item\label{item.ring.proj} ${}_{R}^{g}\mathrm{Mod}$ has no non-zero projective objects.
  \item\label{item.ring.ab4d} ${}_{R}^{g}\mathrm{Mod}$ does not satisfy \textnormal{Ab4*}.
  \end{enumerate}
\end{theorem}

The proof of these are parallel to the proofs of \Cref{th.iff-fin} and \Cref{pr.not-ab4} but needs some modifications. We first rephrase the condition \Cref{eq.cond.ideal}.

\begin{lemma}\label{lem.equiv.cond}
  The hypothesis of \Cref{thm.ring} is equivalent to the following condition: For any $I,K\in\cI$ with $I\subseteq K$ and $0\neq x\in(R/I)_{K}$, there exists $J\in\cI$ with $J\subseteq I$ such that $x$ does not belong to the image of the canonical morphism
  \begin{equation*}
    (R/J)_{K}\to (R/I)_{K}.
  \end{equation*}
\end{lemma}
\begin{proof}
  For ideals $J\subseteq I\subseteq K$, the commutative diagram
  \begin{equation*}
    \begin{tikzcd}
      (R/J)_{K}\ar[r]\ar[d,hook] & (R/I)_{K}\ar[d,hook]\\
      (R/J)_{I}\ar[r] & (R/I)_{I}
    \end{tikzcd}
  \end{equation*}
  shows that any element of $(R/I)_{I}$ not in the image of $(R/J)_{I}\to(R/I)_{I}$ does not belong to the image of $(R/J)_{K}\to(R/I)_{K}$. Thus the condition is equivalent to that with $K=I$.
  
  Since $(R/J)_{I}=(J:I)/J$ and $(R/I)_{I}=R/I$, the image of $(R/J)_{I}\to(R/I)_{I}$ is $((J:I)+I)/I$. Thus the condition is equivalent to
  \begin{equation*}
    \bigcap_{J}\frac{(J:I)+I}{I}=0,
  \end{equation*}
  which is equivalent to \Cref{eq.cond.ideal}.
\end{proof}

\begin{remark}
  The process of proving the vanishing of the limit $\varprojlim_i E_i$ in \Cref{th.iff-fin} shows in particular that for any $N_{i},N_{j},K\in\cN$ with $N_{j}\leq N_{i}$, the image of $\bZ[G/N_{j}]^{K}\to\bZ[G/N_{i}]^{K}$ consists of $[K\cap N_{i}:K\cap N_{j}]$-multiples and the index grows indefinitely when $N_{j}$ gets smaller. With the terminology of \Cref{rem.ex.profin}, this implies that for every $0\neq x\in R/I_{N_{i}}$, there exists some $j$ such that $x$ is not contained in the image of $(R/I_{N_{j}})_{I_{K}}\to(R/I_{N_{i}})_{I_{K}}$. Thus the condition in \Cref{lem.equiv.cond} is satisfied and hence the hypothesis of \Cref{thm.ring} holds.
\end{remark}

\pf{thm.ring}
\begin{thm.ring}
  \Cref{item.ring.proj} Assume that ${}_{R}^{g}\mathrm{Mod}$ has a nonzero projective object $P$. Let $0\neq x\in P$ and take $K\in\cI$ such that $x\in P_{K}$. Similarly to the proof of \Cref{th.iff-fin}, there is an epimorphism $F=\bigoplus_{I\in\cI}F(I)\to P$ where $F(I)$ is a free $R/I$-module. We can assume that $F(I)\neq 0$ only if $I\subseteq K$.
  
  Since $P$ is projective, the epimorphism $F\to P$ splits. We fix a section $P\to F$ and let $\widetilde{x}\in F$ be the image of $x$ by the section. There are finitely many summands $R/I_{1},\ldots,R/I_{n}$ of $F$ such that
  \begin{equation}\label{eq.x.tilde}
    \widetilde{x}=\sum_{j}\widetilde{x}_{j}\in\bigoplus_{j=1}^{n}R/I_{j}
  \end{equation}
  where $0\neq\widetilde{x}\in R_{I_{j}}$. Since $F(I_{j})\neq 0$, we have $I_{j}\subseteq K$ for all $j$.
  
  Applying the condition in \Cref{lem.equiv.cond}, we obtain $J_{1},\ldots,J_{n}\in\cI$ with $J_{j}\subseteq I_{j}$ such that $\widetilde{x}_{j}$ does not belong to the image of
  \begin{equation*}
    (R/J_{j})_{K}\to(R/I_{j})_{K}.
  \end{equation*}
  We construct a module $E$ from $F$ by substituting $R/J_{j}$ for $R/I_{j}$. The section $P\to F$ lifts along the canonical epimorphism $E\to F$, and we obtain a commutative diagram
  \begin{equation*}
    \begin{tikzpicture}[auto,baseline=(current  bounding  box.center)]
      \path[anchor=base] (0,0) node (pl) {$P_{K}$} +(2,.5) node (lim) {$E_{K}$} +(4,0) node (f) {$F_{K}\rlap{.}$};
      \draw[->] (pl) to[bend right=6] node[pos=.5,auto,swap] {$\scriptstyle $} (f);
      \draw[->] (pl) to[bend left=6] node[pos=.5,auto] {$\scriptstyle$} (lim);
      \draw[->] (lim) to[bend left=6] node[pos=.5,auto] {$\scriptstyle$} (f);
    \end{tikzpicture}
  \end{equation*}
  However, $\widetilde{x}$, the image of $x$ along $P_{K}\to F_{K}$, does not belong to the image of $E_{K}\to F_{K}$. This is a contradiction.
  
  \Cref{item.ring.ab4d} We fix $I\in\cI$ until the end of the proof. We claim that the product of the morphisms
  \begin{equation*}
    R/J\to R/I
  \end{equation*}
  where $J\in\cI$ with $J\subseteq I$, is not an epimorphism. For each $K\in\cI$ with $K\subseteq I$, the condition in \Cref{lem.equiv.cond} implies that there is $J\in\cI$ with $J\subseteq I$ such that the image of
  \begin{equation*}
    (R/J)_{K}\to(R/I)_{K}
  \end{equation*}
  does not contain $1\in R/I=(R/I)_{K}$. The product $\prod_{J}(R/J)$ in ${}_{R}^{g}\mathrm{Mod}$ is the direct limit of $(\prod^{f}_{J}(R/J))_{K}$, where $\prod^{f}$ denotes the product in ${}_{R}\mathrm{Mod}$ and $K$ runs over all $K\in\cI$ with $K\subseteq I$. Thus the all-$1$ element of $\prod_{J}(R/I)$ does not belong to the image of $\prod_{J}(R/J)$.
\end{thm.ring}



\begin{thebibliography}{{Kan}18}

\bibitem[AF92]{af}
Frank~W. Anderson and Kent~R. Fuller, \emph{Rings and categories of modules},
  second ed., Graduate Texts in Mathematics, vol.~13, Springer-Verlag, New
  York, 1992. \MR{1245487}

\bibitem[Bre97]{bred}
Glen~E. Bredon, \emph{Sheaf theory}, second ed., Graduate Texts in Mathematics,
  vol. 170, Springer-Verlag, New York, 1997. \MR{1481706}

\bibitem[DNR01]{dnr}
Sorin D\u{a}sc\u{a}lescu, Constantin N\u{a}st\u{a}sescu, and \c{S}erban Raianu,
  \emph{Hopf algebras}, Monographs and Textbooks in Pure and Applied
  Mathematics, vol. 235, Marcel Dekker, Inc., New York, 2001, An introduction.
  \MR{1786197}

\bibitem[Har77]{hrt}
Robin Hartshorne, \emph{Algebraic geometry}, Springer-Verlag, New
  York-Heidelberg, 1977, Graduate Texts in Mathematics, No. 52. \MR{0463157}

\bibitem[{Kan}18]{Kanda}
Ryo {Kanda}, \emph{{Non-exactness of direct products of quasi-coherent
  sheaves}}, arXiv:1810.08752.

\bibitem[RZ10]{rz}
Luis Ribes and Pavel Zalesskii, \emph{Profinite groups}, second ed., Ergebnisse
  der Mathematik und ihrer Grenzgebiete. 3. Folge. A Series of Modern Surveys
  in Mathematics [Results in Mathematics and Related Areas. 3rd Series. A
  Series of Modern Surveys in Mathematics], vol.~40, Springer-Verlag, Berlin,
  2010. \MR{2599132}

\bibitem[Ser02]{ser-gal}
Jean-Pierre Serre, \emph{Galois cohomology}, english ed., Springer Monographs
  in Mathematics, Springer-Verlag, Berlin, 2002, Translated from the French by
  Patrick Ion and revised by the author. \MR{1867431}

\bibitem[Ste75]{Stenstrom}
Bo~Stenstr\"{o}m, \emph{Rings of quotients}, Springer-Verlag, New
  York-Heidelberg, 1975, Die Grundlehren der Mathematischen Wissenschaften,
  Band 217, An introduction to methods of ring theory. \MR{0389953}

\bibitem[Wil98]{Wilson}
John~S. Wilson, \emph{Profinite groups}, London Mathematical Society
  Monographs. New Series, vol.~19, The Clarendon Press, Oxford University
  Press, New York, 1998. \MR{1691054}

\end{thebibliography}
\bibliographystyle{customamsalpha}
\addcontentsline{toc}{section}{References}

\providecommand{\bysame}{\leavevmode\hbox to3em{\hrulefill}\thinspace}
\providecommand{\MR}{\relax\ifhmode\unskip\space\fi MR }
\providecommand{\MRhref}[2]{%
  \href{http://www.ams.org/mathscinet-getitem?mr=#1}{#2}
}
\providecommand{\href}[2]{#2}

\Addresses

\end{document}